\documentclass[reqno,12pt]{amsart}
\usepackage{graphicx}
\usepackage[utf8]{inputenc}
\usepackage{tikz,amsthm,amsmath,amssymb,mathtools}
\usepackage{amsmath}
\usepackage{amsfonts}
\usepackage{amssymb}
\usepackage{pdfsync}
\usepackage{color,graphicx,tikz-cd}
\usepackage{a4wide}
\usepackage{amscd,amsthm,latexsym}
\usepackage{tikz,mathtools}
\usepackage[colorlinks=true]{hyperref}
\usepackage{relsize}
\usepackage{tikz}
\usetikzlibrary{math}
\usepackage{float}
\usepackage[misc]{ifsym}
\usepackage{caption}
\usepackage{enumitem}
\usepackage{bm}
\usepackage{graphicx}
\usepackage{ifthen}
\usepackage{tikz-3dplot}
\usepackage{fullpage}
\numberwithin{equation}{section} 
\usepackage[utf8]{inputenc}
\usepackage{kotex}
\usepackage{kotex-logo}
\usepackage{soul}
\usepackage{comment}
\usepackage{tikz-3dplot}
\usepackage{color,graphicx,tikz-cd}
\usepackage[colorinlistoftodos, textwidth = 2.3cm]{todonotes}
\usetikzlibrary{decorations.markings}
\usetikzlibrary[cd]
\usetikzlibrary{arrows}
\usetikzlibrary{calc}
\graphicspath{}

\title{Dualities and reciprocities on graphs on surfaces}

\author{Woo-Seok Jung}
\address{Department of Mathematics, Sogang University, Seoul, South Korea}
\email{jungws@sogang.ac.kr}

\author{Jaeseong Oh}
\address{Department of Mathematics, Seoul National University, Seoul,
South Korea}
\email{jaeseong\_oh@snu.ac.kr}

\date{\today}

\keywords{Graph on a surface, Tension, Flow, Acyclic orientation, Totally cyclic orientation, Duality, Ehrhart theory, Combinatorial reciprocity}

\subjclass[2010]{Primary: 05C10, 05C30, 05C31}

\date{\today}

\newtheorem{thm}{Theorem}[section]
\newtheorem{prop}[thm]{Proposition}
\newtheorem{lem}[thm]{Lemma}
\newtheorem{cor}[thm]{Corollary}

\theoremstyle{definition}
\newtheorem{defn}[thm]{Definition}
\newtheorem{ex}[thm]{Example}
\newtheorem{rem}[thm]{Remark}

\newcommand\GG{\mathbb{G}}

\newcommand{\sslash}{\mathbin{/\mkern-6mu/}}
\newcommand{\ssetminus}{\mathbin{\setminus\mkern-6mu\setminus}}

\usetikzlibrary{decorations.pathreplacing}
\newcommand{\indsize}{\scriptsize}
\newcommand{\colind}[2]{\displaystyle\smash{\mathop{#1}^{\raisebox{.5\normalbaselineskip}{\indsize #2}}}}
\newcommand{\rowind}[1]{\mbox{\indsize #1}}

\definecolor{green1}{rgb}{0,0.5,0}
\definecolor{blue1}{rgb}{0,0,0.7}

\newenvironment{red1}{\relax\color{red}}{\hspace*{.5ex}\relax}
\newenvironment{blue1}{\relax\color{blue1}}{\hspace*{.5ex}\relax}
\newenvironment{green1}{\relax\color{green1}}{\hspace*{.5ex}\relax}
\newenvironment{magenta}{\relax\color{magenta}}{\hspace*{.5ex}\relax}
\newenvironment{purple}{\relax\color{purple}}{\hspace*{.5ex}\relax}

\newcommand{\bema}{\begin{magenta}}
\newcommand{\ema}{\end{magenta}}

\newcommand{\bepu}{\begin{purple}}
\newcommand{\epu}{\end{purple}}

\begin{document}

\maketitle



\begin{abstract}
We extend the duality between acyclic orientations and totally cyclic orientations on planar graphs to dualities on graphs on orientable surfaces by introducing boundary acyclic orientations and totally bi-walkable orientations. In addition, we provide a reciprocity theorem connecting local tensions and boundary acyclic orientations. Furthermore, we define the balanced flow polynomial which is connected with tension polynomial by duality and with totally bi-walkable orientations by reciprocity.
\end{abstract}

\section{Introduction}
Enumeration of tensions, flows, acyclic orientations, and totally cyclic orientations of graphs is a classical topic in enumerative combinatorics. These are related by `planar duality' and `reciprocity'. 

A \emph{(cellularly embedded) graph $G$ on a surface $\Sigma$} is an embedding of a graph $G$ into a closed orientable surface $\Sigma$ such that $\Sigma\setminus G$ is a disjoint union of disks. In this paper, we study dualities and reciprocities on graphs on surfaces.

\subsection{The dualities}
Let $G$ be an (arbitrarily) graph with an orientation $\sigma$ and $\mathbb{Z}_k$ be a cyclic group of order $k$. A $\mathbb{Z}_k$-\emph{tension} of $G$ is a mapping $x:E\overrightarrow{}\mathbb{Z}_k$ such that for each directed cycle $\mathcal{C}$ in $G$, 
$$\sum_{e\in C}s(e)x(e)=0,$$
where $s(e)=1$ if the orientation of $e$ in $\mathcal{C}$ and $\sigma$ agree and $s(e)=-1$ otherwise.
A $\mathbb{Z}_k$-\emph{flow} of $G$ is a mapping $x:E\overrightarrow{} \mathbb{Z}_k$ such that for each vertex $v$ in $G$,
$$\sum_{h(e)=v}x(e)=\sum_{t(e)=v}x(e),$$ where $h(e)$ and $t(e)$ denote the head and tail of an edge $e$ in $\sigma$.
Tutte \cite{Tut47} showed that both the number of nowhere-zero $\mathbb{Z}_k$-tension of $G$ and the number of nowhere-zero $\mathbb{Z}_k$-flow are polynomials in $k$. The \emph{tension polynomial} $\tau(G;k)$ counts nowhere-zero $\mathbb{Z}_k$-tensions and the \emph{flow polynomial} $\phi(G;k)$ counts nowhere-zero $\mathbb{Z}_k$-flows. For a planar graph $G$ and its dual graph $G^{*}$, there is a tension-flow duality, which says that $$\phi(G;k)=\tau(G^*;k).$$ This might be one of the main attractions for Tutte to define the \emph{Tutte polynomial}. Similar to the way the dual graph is defined for a planar graph, the dual graph can be naturally defined for a graph on a surface. However, the tension-flow duality fails for graphs on surfaces. For a graph $G$ on a torus in Figure \ref{fig:dual}, $$\tau(G;k)=0\neq (k-1)^2=\phi(G^*;k).$$ In \cite{DGMVZ05}, DeVos et al. defined the \emph{local tension} on graphs on surfaces and described the duality between local tensions and flows. In this paper, we define the \emph{balanced flow} which is dual of the tension.

An \emph{acyclic orientation} of $G$ is an orientation with no directed cycle. A \emph{totally cyclic orientation} of $G$ is an orientation such that for each edge $e$ in $G$, there is a directed cycle containing $e$. There is a duality between acyclic orientations and totally cyclic orientations on planar graphs. This duality also fails for graphs on surfaces, in general. In Figure \ref{fig:dual}, there is no acyclic orientation of $G$, while there are $4$ totally cyclic orientations of $G^*$. To extend the duality of acyclic orientations and totally cyclic orientations for planar graphs to graphs on surfaces, we introduce \emph{boundary acyclic orientations} and \emph{totally bi-walkable orientations}, which are dual of totally cyclic orientations and acyclic orientations, respectively.

\begin{figure}[H]
\resizebox{0.6\textwidth}{!}{
\begin{tabular}{c c}

\begin{tikzpicture}
  \useasboundingbox (-3,-1.5) rectangle (3,1.5);
  \draw (0,0) ellipse (3 and 1.5);
  \begin{scope}
    \clip (0,-1.8) ellipse (3 and 2.5);
    \draw (0,2.2) ellipse (3 and 2.5);
  \end{scope}
  \begin{scope}
    \clip (0,2.2) ellipse (3 and 2.5);
    \draw (0,-2.2) ellipse (3 and 2.5);
  \end{scope}
  
\draw[thick] (0,0) ellipse (2.3 and 1);
\draw[thick, dashed] (-1,-1.4) arc (-90:90:.4 and 0.61);
\draw[thick] (-1,-1.4) arc (270:90:.5 and .61);

\node[circle,fill,inner sep=1pt] at (-1.5,-0.75) {};
\end{tikzpicture}
&
\begin{tikzpicture}
  \useasboundingbox (-3,-1.5) rectangle (3,1.5);
  \draw (0,0) ellipse (3 and 1.5);
  \begin{scope}
    \clip (0,-1.8) ellipse (3 and 2.5);
    \draw (0,2.2) ellipse (3 and 2.5);
  \end{scope}
  \begin{scope}
    \clip (0,2.2) ellipse (3 and 2.5);
    \draw (0,-2.2) ellipse (3 and 2.5);
  \end{scope}
  
\draw[thick,gray!60] (0,0) ellipse (2.3 and 1);

\draw[thick, gray!60,dashed] (-1,-1.4) arc (-90:90:.4 and 0.61);
\draw[thick,gray!60] (-1,-1.4) arc (270:90:.5 and .61);
\node[circle,fill=gray!60,inner sep=1pt] at (-1.5,-0.75) {};

\draw[thick,red] (0,0) ellipse (2 and 0.7);
\draw[thick,red] (1,-1.4) arc (-90:90:.5 and .61);
\draw[thick, red,dashed] (1,-1.4) arc (270:90:.4 and 0.61);
\node[circle,fill=red,inner sep=1pt] at (1.44,-0.49) {};
\end{tikzpicture}
\end{tabular}
}
    \caption{A graph $G$(black) on a torus its dual graph $G^*$(red).} \label{fig:dual}
\end{figure}
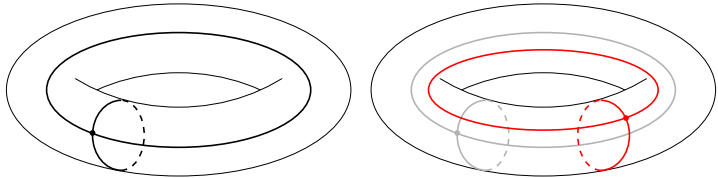

\subsection{The reciprocities}
Another main purpose of this paper is to present a reciprocity theorem connecting local tension polynomials and boundary acyclic orientations. Before explaining our work, we give a brief history of studying reciprocity on the enumerations of coloring and flows of graphs. 
Stanley \cite[Theorem 1.2]{Sta73} provided the first example of combinatorial reciprocity, which states that the absolute value of the chromatic polynomial of a graph evaluated at $-k$ is equal to the number of pairs $(c, \mathfrak{o})$, where $c$ is a coloring and $\mathfrak{o}$ is an acyclic orientation compatible with $c$.
By exploiting the Ehrhart theory, Breuer and Sanyal proved reciprocity theorems for flow polynomials \cite[Theorem 4.2]{BS12} and tension polynomials \cite[Theorem 2.2]{BS12}.

Following the idea of \cite{BS12}, we provide Theorem \ref{reciprocity for local tension} which relates local tensions and boundary acyclic orientations on graphs on surfaces. We also show that the balanced flows are related to totally bi-walkable orientations by reciprocity.
In addition, we present a reciprocity theorem for integral $k$-local tensions (defined in Section~\ref{section4.2}) on graphs on surfaces.

\subsection{Organization of the paper}
We begin in Section~\ref{section2} by giving basic definitions and properties of graphs on surfaces. In Section~\ref{section3}, we define the boundary acyclic orientation and the totally bi-walkable orientation of a graph on a surface and prove two duality theorems (Theorem \ref{duality of local AO and TCO} and Theorem \ref{duality of AO and TBO}) each of which involves those orientations. For application, we count the number of totally bi-walkable orientations under some conditions. 
In Section~\ref{section4}, we provide a proof of a reciprocity theorem for local tension polynomials (Theorem~\ref{reciprocity for local tension}). We also define the balanced flow and prove a reciprocity theorem for balanced flow polynomials. In the subsection~\ref{section4.2}, we prove a reciprocity theorem for integral local tensions.

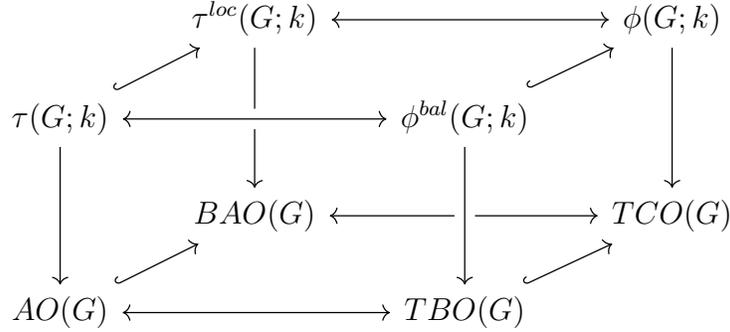
\begin{figure}[H]\label{cube}
\begin{tikzcd}[row sep=scriptsize, column sep=scriptsize]
& \tau^{loc}(G;k) \arrow[rr, leftrightarrow] \arrow[dd] & & \phi(G;k) \arrow[dd] \\
\tau(G;k) \arrow[ur, hook] \arrow[rr, leftrightarrow, crossing over] \arrow[dd] & & \phi^{bal}(G;k) \arrow[ur, hook]\\
& BAO(G)  \arrow[rr, leftrightarrow] & & TCO(G)  \\
AO(G) \arrow[ur, hook] \arrow[rr, leftrightarrow] & & TBO(G) \arrow[ur, hook] \arrow[from=uu, crossing over]\\
\end{tikzcd}
    \caption{Dualities and reciprocities on graphs on surfaces.}
    \label{figure:cube}
\end{figure}

Figure \ref{figure:cube} summarizes main results of this paper. In the Figure, $\tau^{loc}(G;k)$ and $\phi^{bal}(G;k)$ represent the local tension polynomial and the balanced flow polynomial of $G$, respectively. And, $AO(G)$, $BAO(G)$, $TBO(G)$, and $TCO(G)$ denote the set of acyclic orientations, boundary acyclic orientations, totally bi-walkable orientations, and totally cyclic orientations of $G$, respectively. Dualities are represented by the horizontal lines. The vertical lines represent that the number of objects below can be obtained by specializing the polynomial above at $-1$ up to sign. Moreover, those objects are related by the reciprocities. The diagonal lines represent inclusions among corresponding objects (or the objects that are counted by given polynomials).


\section{Preliminaries of Graphs on surfaces and ribbon graphs}\label{section2}
In this section, we give a brief overview of the basic definitions and properties of graphs on surfaces. A reader familiar with this topic may skip or skim this section. For details, we recommend \cite{EM13}.

A (abstract) \emph{graph} $G=(V,E)$ is given by a set of vertices $V$ and a set of edges $E$ with incidence relations between $V$ and $E$ such that $e\in E$ is either incident to two distinct vertices $v,w\in V$ or a single vertex $v\in V$. For the latter case, $e$ is called a loop.

Let $G=(V,E)$ be a graph and $\Sigma$ be a closed surface. A \emph{graph embedding} of $G$ on $\Sigma$ is a representation of the vertices of $G$ as distinct points in $\Sigma$ and the edges of $G$ as arcs connecting the points associated with vertices. In addition, this representation should give a topological embedding of $G$ in $\Sigma$, i.e., there is no intersection between edges or a vertex contained in the interior of an arc. Throughout this paper for a graph $G$ on a surface $\Sigma$, we assume that the surface $\Sigma$ is orientable and the graph $G$ is \emph{cellularly embedded} on $\Sigma$. By `cellularly embedded' we mean that the complement $\Sigma\setminus G$ is a disjoint union of discs each of which is called the \emph{face}. A graph on a surface is denoted by $G=(V,E,F)$ where $V,E,$ and $F$ are the set of vertices, edges, and faces of $G$, respectively. We denote the number of connected components of $\Sigma$ by $c(G\subseteq \Sigma)=c(G)$ and say $G$ is connected if $c(G)=1$. A graph $G$ on a surface $\Sigma$ is called \emph{planar} if $\Sigma$ is a disjoint union of spheres.

A graph on a surface can be equivalently described as a \emph{ribbon graph}. A ribbon graph $\GG=(V,E)$ is a surface with boundaries which is a union of two sets $V$ and $E$ of discs such that
\begin{enumerate}
    \item the vertices and edges intersect in disjoint line segments,
    \item each such line segment lies on the boundary of precisely one vertex and precisely one edge and
    \item every edge contains exactly two such line segments.
\end{enumerate}
Throughout this paper, we assume that ribbon graphs are orientable surfaces with boundaries. Note that we get a closed orientable surface if we glue a disk at each boundary component of $\GG$. The glued disks are called \emph{faces} of $\GG$, and we will denote a ribbon graph $\GG$ by $\GG=(V,E,F)$ where $V,E$, and $F$ are the set of vertices, edges, and faces of $\GG$, respectively.

Two concepts, graphs on surfaces and ribbon graphs, might seem different at first glance but they are describing the same objects. To see this, let $G$ be a graph on a surface  $\Sigma$. One can obtain a ribbon graph by `fattening' vertices and edges of $G$ in $\Sigma$ and forget the faces. By this identification, ribbon graphs are sometimes called \emph{fat graphs}. Conversely, one can get a graph on a surface from a ribbon graph by the other way around. This procedure is described in Figure \ref{figure:graph on surface and ribbon graph}. Via this correspondence, operations and properties of graphs on surfaces can be understood as those of ribbon graphs and vice versa.

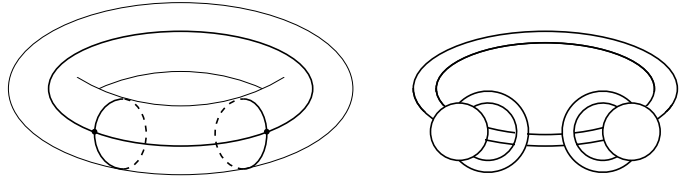
\begin{figure}[H]
\resizebox{0.6\textwidth}{!}{
\begin{tabular}{c c}

\begin{tikzpicture}
  \useasboundingbox (-3,-1.5) rectangle (3,1.5);
  \draw (0,0) ellipse (3 and 1.5);
  \begin{scope}
    \clip (0,-1.8) ellipse (3 and 2.5);
    \draw (0,2.2) ellipse (3 and 2.5);
  \end{scope}
  \begin{scope}
    \clip (0,2.2) ellipse (3 and 2.5);
    \draw (0,-2.2) ellipse (3 and 2.5);
  \end{scope}
  
\draw[thick] (0,0) ellipse (2.3 and 1);
\draw[thick, dashed] (-1,-1.4) arc (-90:90:.4 and 0.61);
\draw[thick] (-1,-1.4) arc (270:90:.5 and .61);
\draw[thick] (1.1,-1.4) arc (-90:90:.4 and 0.61);
\draw[thick, dashed] (1.1,-1.4) arc (270:90:.5 and .61);
\node[circle,fill,inner sep=1pt] at (1.5,-0.75) {};
\node[circle,fill,inner sep=1pt] at (-1.5,-0.75) {};
\end{tikzpicture}

&

\begin{tikzpicture}
  \useasboundingbox (-3,-1.5) rectangle (3,1.5);
\begin{scope}
    \clip (-3,-0.4) rectangle (3,5);
    \draw[thick] (0,0) ellipse (2.3 and 1);
\end{scope}

\begin{scope}
    \clip (-3,-3) rectangle (-1.95,0);
    \draw[thick] (0,0) ellipse (2.3 and 1);
\end{scope}

\begin{scope}
    \clip (-1.04,-3) rectangle (-0.53,0);
    \draw[thick] (0,0) ellipse (2.3 and 1);
\end{scope}

\begin{scope}
    \clip (-0.31,-3) rectangle (0.32,0);
    \draw[thick] (0,0) ellipse (2.3 and 1);
\end{scope}

\begin{scope}
    \clip (0.55,-3) rectangle (1.01,0);
    \draw[thick] (0,0) ellipse (2.3 and 1);
\end{scope}

\begin{scope}
    \clip (1.95,-1) rectangle (2.3,-0.3);
    \draw[thick] (0,0) ellipse (2.3 and 1);
\end{scope}

---
\begin{scope}
    \clip (-3,-0.05) rectangle (3,5);
    \draw[thick] (0,0) ellipse (1.9 and 0.8);
\end{scope}

\begin{scope}
    \clip (-3,-3) rectangle (-1.75,0);
    \draw[thick] (0,0) ellipse (1.9 and 0.8);
\end{scope}

\begin{scope}
    \clip (-3,-0.3) rectangle (3,3);
    \draw[thick] (0,0) ellipse (1.9 and 0.8);
\end{scope}

\begin{scope}
    \clip (-1.0,-3) rectangle (-0.53,0);
    \draw[thick] (0,0) ellipse (1.9 and 0.8);
\end{scope}

\begin{scope}
    \clip (-0.31,-3) rectangle (0.3,0);
    \draw[thick] (0,0) ellipse (1.9 and 0.8);
\end{scope}

\begin{scope}
    \clip (0.5,-3) rectangle (1.02,0);
    \draw[thick] (0,0) ellipse (1.9 and 0.8);
\end{scope}

\begin{scope}
    \clip (1.9,-1) rectangle (2.5,-0.3);
    \draw[thick] (0,0) ellipse (1.9 and 0.8);
\end{scope}

\draw[thick] (1.25,-0.75+0.866*0.5) arc (60:300:0.5);
\draw[thick] (1.5,-0.25) arc (45:315:1.414*0.5);
\draw[thick] (1.5,-0.75) circle (0.5); 
\draw[thick] (-1.5,-0.75) circle (0.5); 
\draw[thick] (-1.5,-0.25) arc (135:-135:1.414*0.5); 
\draw[thick] (-1.25,-0.75+0.866*0.5) arc (120:-120:0.5); 

\end{tikzpicture}
\end{tabular}
}
    \caption{A graph on a surface and its corresponding ribbon graph.}
    \label{figure:graph on surface and ribbon graph}
\end{figure}

Given a ribbon graph $\GG$ and an edge $e$ in $\GG$, there are two elementary but important operations called \emph{deletion} and \emph{contraction} of $e$. The deletion of $e$ in $\GG$, denoted by $\GG\setminus e$, is given by deleting the edge $e$ from the original graph $\GG$. To define contraction, consider two cases. For a non-loop edge $e$, assume that $e$ is incident to distinct vertices $v$ and $w$. The contraction $\GG/e$ is given by replacing three discs $e$, $v$, and $w$ by a single disc $e\cup v\cup w$ as a new vertex. This operation coincide with the contraction for abstract graphs. For a loop $e$ incident to a vertex $v$, the union $e\cup v$ is an annulus. The contraction $\GG/e$ is obtained by replacing two discs $e$ and $v$ by two new vertices given by two discs that bounds the boundary of $e\cup v$. Note that this operation does not coincide with the contraction for abstract graphs.
Applying a sequence of deletions and contractions, the order of operation is immaterial \cite[p. 923]{GKRV18}. Therefore, for disjoint subsets of edges $A, B\subseteq E$, contraction of edges in $A$ and deletion of edges in $B$ from $G$ is well defined and it is denoted by $\GG/A\setminus B$.
For a graph on a surface $G$, one can then define the operations of deletion and contraction on $G$ via the identification of graphs on surfaces and ribbon graphs. Note that a deletion and contraction on $G$ may change the surface $\Sigma$ in which $G$ is embedded. For example, $G$ is embedded in a torus in Figure \ref{figure:graph on surface and ribbon graph} and deletion of any two edges results in a planar graph.

\begin{rem}
Topologically, the contraction of loop $e$ attached to a vertex $v$ can be described as follow: Consider the surface with boundary $\Sigma\setminus (e\cup v)$. The boundary of this surface is given by two circles. By `contracting' the two circles, $\Sigma\setminus (e\cup v)$ is homotopy equivalent to a two punctured surface. The contraction of $e$ is obtained by filling the two punctures with two new vertices to get a closed surface and graph embedded in the surface.
\end{rem}

There is another important notion, the \emph{dual graph}. This operation switches the role of vertices and faces. To be more precise, the dual of a ribbon graph $\GG$ is obtained from $\GG$ by attaching new vertices (discs) at each connected component of the boundary of $\GG$, deleting the original vertices and considering the edges attached to the new vertices as \emph{dual edges}. For a planar ribbon graph $\GG$, this operation coincides with the usual planar dual. The dual graph of $\GG$ is denoted by $\GG^*$. Unlike  deletion and contraction, the dual operation does not change the surface by construction. Obviously, one can define dual operation for a graph on a surface $G$ via identification of graphs on surfaces and ribbon graphs. We end this section with a theorem which states that the dual operation exchanges deletion and contraction of graphs on surfaces.

\begin{prop} \cite[Chapter 4]{EM13} \label{del-con dual}
For a ribbon graph $\GG=(V,E,F)$ and $A\subseteq E$, deletion and contraction are dual to each others, i.e.,
$${\GG^*}\setminus A=(\GG/A)^*.$$
\end{prop}


\section{The duality theorems}\label{section3}

Let $G$ be a graph. An orientation of $G$ is an assignment of direction to each edge of $G$. There are two important classes of orientations given in the following. An orientation $\mathfrak{o}$ is called an \emph{acyclic orientation} if there is no directed cycle in $\mathfrak{o}$. An orientation $\mathfrak{o}$ is called a \emph{totally cyclic orientation} if for each edge $e$, there exists a directed cycle in $\mathfrak{o}$ containing $e$. Since the definition of orientation, acyclic orientation, and totally cyclic orientation do not depend on the embedding, we shall define those for graphs on surfaces in the same way.

From now on, every surface is assumed to be oriented. Let $G$ be a graph on $\Sigma$ and $\mathfrak{o}$ be an orientation of $G$. The \emph{dual orientation} $\mathfrak{o}^{*}$ of $\mathfrak{o}$ is an orientation of $G^{*}$ so that each pair $(\overrightarrow{e},\overrightarrow{e}^{*})$ of oriented edges of $G$ and $G^*$ forms a positively oriented ordered pair with respect to the given orientation of the surface $\Sigma$. Figure \ref{figure:dual orientation} describes an orientation and its dual orientation of a graph (on a sphere). As mentioned in the introduction, the number of acyclic orientations of a planar graph $G$ equals to the number of totally cyclic orientations of the dual graph $G^*$. Moreover, the operation of dual orientation gives a bijection between those objects.

\begin{figure}[H]
\begin{tikzpicture}
[scale=0.8,every node/.style={transform shape}] 
\fill (0,1) circle (2pt); 
\fill (0,-1) circle (2pt); 
\fill ({sqrt(3)},0) circle (2pt); 
\fill (-{sqrt(3)},0) circle (2pt); 
\begin{scope}
    [decoration={
    markings,
    mark=at position 1 with {\arrow{latex}}}
    ] 
\draw[postaction={decorate}] (0,1) -- (-{sqrt(3)},0); 
\draw[postaction={decorate}] (-{sqrt(3)},0) -- (0,-1); 
\draw[postaction={decorate}] (0,1) -- (0,-1); 
\draw[postaction={decorate}] ({sqrt(3)},0) --(0,1); 
\draw[postaction={decorate}] ({sqrt(3)},0) --(0,-1);
\end{scope}

\fill[red] (-0.5,0) circle (2pt);
\fill[red] (0.5,0) circle (2pt);
\fill[red] (0,{1+sqrt(5)/2}) circle (2pt);

\begin{scope}
    [decoration={
    markings,
    mark=at position 1 with {\arrow{latex}}}
    ] 
\draw[postaction={decorate}, red, thick] (-0.5,0) -- (0.5,0);

\tikzset{
    partial ellipse/.style args={#1:#2:#3}{
        insert path={+ (#1:#3) arc (#1:#2:#3)}
    }
}
\draw[thick, red, -latex] (0,0.5) [partial ellipse=90:-20:0.5 and {sqrt(5)/2+0.5}];
\draw[thick, red, -latex] (0,0.5) [partial ellipse=90:200:0.5 and {sqrt(5)/2+0.5}];
\draw[thick, red, -latex] (0,0) [partial ellipse=0:90:2.5 and {1+sqrt(5)/2}];
\draw[thick, red] (-2.5,0) arc (180:90: 2.5 and {1+sqrt(5)/2});
\draw[thick, red, -latex] (-1.5,0) [partial ellipse=-180:0:1 and 0.5];
\draw[thick, red] (0.5,0) arc (-180:0: 1 and 0.5);

\end{scope}

\end{tikzpicture}
    \caption{An orientation of a graph (black) and its dual orientation of a dual graph (red).}
    \label{figure:dual orientation}
\end{figure}
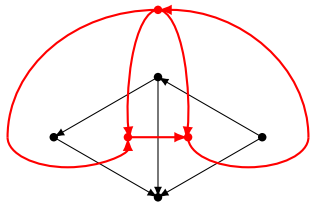

\begin{thm} \cite{Noy01}
For a planar graph $G$, the map $AO(G) \to TCO(G^{*})$ sending $\mathfrak{o}$ to $\mathfrak{o}^{*}$ is bijective.
\end{thm}

Also as pointed out, the duality between acyclic orientations and totally cyclic orientations does not hold for graphs on surfaces, in general. Then it is natural to ask how to characterize the `dual orientation' of acyclic orientations or totally cyclic orientations. In this section, we introduce boundary acyclic orientations and totally bi-walkable orientations and provide two duality theorems involving these orientations.

\subsection{The dual of totally cyclic orientations}
Let $G=(V,E,F)$ be a graph on a surface. For a subset $F'\subseteq F$, the boundary $\partial F'$ of $F'$ is the set of edges having two distinct adjacent faces only one of which belongs to $F'$. For example in Figure \ref{figure:kite} for $F'=\{f_1,f_3\}$, $\partial F'=\{ e_3,e_4,e_5,e_6 \}$. Since two adjacent faces of $e_1$ are all in $F'$, it is not contained in $\partial F'$. A subset of $E$ is called a \emph{boundary} if it is a boundary of a subset of $F$. For a boundary $\partial F'$, there is a natural orientation $\overrightarrow{\partial F'}$: Each edge in $\partial F'$ is given a direction induced by the orientation of its adjacent face $f\in F'$ as a submanifold of $\Sigma$. 

Let $\mathfrak{o}$ be an orientation of in $G$. Then $\overrightarrow{\partial_\mathfrak{o} F'}$ is an element of $\{0,\pm1\}^E$ given by the following:
\begin{align*}
 \overrightarrow{\partial_\mathfrak{o} F'}(e) &=
  \begin{cases}
    0,       & \text{if $e$ is not in $\partial F'$}\\
    1,       & \text{if the orientation of $e$ in $\overrightarrow{\partial F'}$ agrees with $\mathfrak{o}$}\\
   -1,       & \text{if the orientation of $e$ in $\overrightarrow{\partial F'}$ disagrees with $\mathfrak{o}$}
  \end{cases}
\end{align*}
For example, let $\mathfrak{o}$ be an orientation given in Figure \ref{figure:kite} and the torus is oriented counter-clockwise. Then we have $\overrightarrow{\partial_\mathfrak{o} \{f_1, f_3\}}=(0,0,1,-1,1,-1)$. We abuse the notation $\partial f$, $\overrightarrow{\partial f}$ and $\overrightarrow{\partial_\mathfrak{o} f}$ when $F'=\{f\}$. Now we are ready to define the \emph{boundary acyclic orientation}.

\begin{figure}[H]
\begin{tikzpicture}
[scale=1.5,every node/.style={transform shape}] 
\fill (1,0) circle (0.5pt); 
\fill (1,2) circle (0.5pt); 
\fill (0,1) circle (0.5pt);
\fill (-1,0) circle (0.5pt); 
\fill (-1,2) circle (0.5pt); 
\begin{scope}
    [decoration={
    markings,
    mark=at position 1 with {\arrow[scale=1.5]{latex}}}
    ] 
\draw[postaction={decorate}] (-1,0) -- (1,0); 
\draw[postaction={decorate}] (-1,0) -- (-1,2); 
\draw[postaction={decorate}] (1,0) -- (1,2); 
\draw[postaction={decorate}] (-1,2) -- (1,2); 
\draw[postaction={decorate}] (-1,0) -- (0,1);
\draw[postaction={decorate}] (-1,2) -- (0,1);
\draw[postaction={decorate}] (1,0) -- (0,1);
\draw[postaction={decorate}] (1,2) -- (0,1);

\end{scope}

\node at (0,2.2) {\tiny{$e_1$}};
\node at (0,-0.2) {\tiny{$e_1$}};
\node at (-1.3,1) {\tiny{$e_2$}};
\node at (1.3,1) {\tiny{$e_2$}};
\node at (-0.5,1.5) {\tiny{$e_3$}};
\node at (-0.5,0.5) {\tiny{$e_4$}};
\node at (0.5,1.5) {\tiny{$e_6$}};
\node at (0.5,0.5) {\tiny{$e_5$}};
\node at (0,1.65) {\tiny{$f_1$}};
\node at (0,0.35) {\tiny{$f_3$}};
\node at (-0.65,1) {\tiny{$f_2$}};
\node at (0.65,1) {\tiny{$f_4$}};

\draw [thick] (0,-.1) -- (0,.1);
\draw [thick] (0,1.9) -- (0,2.1);
\draw [thick] (-1.1,1.05) -- (-0.9,1.05);
\draw [thick] (-1.1,.95) -- (-0.9,.95);
\draw [thick] (0.9,1.05) -- (1.1,1.05);
\draw [thick] (0.9,0.95) -- (1.1,.95);

\end{tikzpicture}
    \caption{An oriented graph on a torus.}
    \label{figure:kite}
\end{figure}
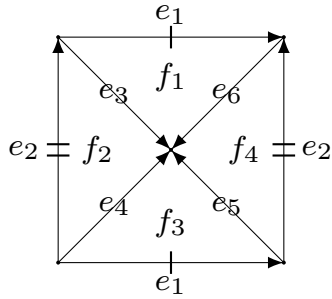

\begin{defn}
Let $G$ be a graph on a surface and let $\mathfrak{o}$ be an orientation of $G$. A boundary of $F'$ is said to be \emph{coherently oriented} (with respect to $\mathfrak{o}$) if the nonzero entries in $\overrightarrow{\partial_\mathfrak{o} F'}$ are all $+1$ or are all $-1$.

An orientation $\mathfrak{o}$ on $G$ is called \emph{boundary acyclic orientation} if there is no coherently oriented boundary in $G$. The set of boundary acyclic orientations is denoted by $BAO(G)$.
\end{defn}

\begin{rem}\label{Remark: boundary is a 'boundary'}
The boundary of $F'$ can be considered as the boundary of a surface $\Sigma'$ naturally induced from $F'$ by the following description. Let $\GG=(V,E,F)$ be the corresponding ribbon graph of $G$. For a subset $F'\subseteq F$, consider the adjacent ribbon edges 
$$E'=\{e \in E : \text{$e$ is adjacent to a face in $F'$} \}$$
and the adjacent ribbon vertices 
$$V'=\{v \in V : \text{$v$ is adjacent to an edge in $E'$} \}$$
of $F'$.
Then $\Sigma'= F' \cup E' \cup V'$ forms a compact surface with boundary and embedded in $\Sigma$, and the boundary $\partial \Sigma'$ of $\Sigma'$ is homeomorphic to a union of disjoint circles.
The boundary $\partial \Sigma'$ consists of some portion of boundaries of ribbon edges and vertices of the ribbon graph. The edges of $G$ whose corresponding ribbon edges meet with $\partial \Sigma'$ is the boundary $\partial F'$ of $F'$. 

From the above argument, $\partial F'$ is a union of some cycles of $G$.
Therefore an acyclic orientation is a boundary acyclic orientation and $AO(G) \subseteq BAO(G)$.
\end{rem}

To give a duality between boundary acyclic orientations and totally cyclic orientations, we present another characterization for totally cyclic orientations. A \emph{cut} of a graph $G$ is a partition $(V_{1}, V_{2})$ of the vertex set $V$ of $G$ into two nonempty disjoint subsets. A cut determines a \emph{cut set}, the set of edges that have one endpoint in each subset of the partition. We abbreviate a non-empty subset of $E$ as a cut set if it is a cut set of a cut of $G$.
An orientation of a cut set is called \emph{coherently oriented} if its elements are oriented from one part toward the other part of the cut.
The following lemma is a well-known characterization of totally cyclic orientations. The argument of the proof will be used in the proof of Lemma~\ref{boundary acyclic orientation}.

\begin{lem}\label{lem tco characterization} 
Let $\mathfrak{o}$ be an orientation of a graph $G$.
\begin{enumerate}[label = {\rm (\arabic*)}]
    \item The orientation $\mathfrak{o}$ is a totally cyclic orientation if and only if $\mathfrak{o}$ has no coherently oriented cut set. 
    \item Suppose $G$ is connected. The orientation $\mathfrak{o}$ is a totally cyclic orientation if and only if $\mathfrak{o}$ is strongly connected, i.e. every vertex is reachable from every other vertex by a directed path.
\end{enumerate}
\end{lem}

\begin{proof}
(1) Suppose that there is a coherently oriented cut set $S$ of $G$. Then any edge $\overrightarrow{e} \in S$, cannot be contained in a directed cycle.

Conversely, suppose there is an oriented edge $\overrightarrow{e}=(v,w)$ oriented toward a vertex $w$ from $v$, which is not contained in any directed cycle. Let $H$ be the subset of $V(G)$ consisting of the vertices reachable by directed paths of $\mathfrak{o}$ starting from $w$, including $w$ itself. Since there is no directed cycle containing $\overrightarrow{e}$, $v$ is in $V(G) \setminus H$. If an edge has one endpoint in $H$ and the other point in $V(G) \setminus H$, it should be oriented toward $H$ by the construction of $H$. Let $S$ be the set of all such edges. Then $S$ is the coherently oriented edge cut set which separates $H$ and $V(G) \setminus H$.

(2) The if part is clear. To prove the only if part, let $v, w$ be arbitrary vertices of $G$ and $K$ be the set of all vertices reachable by a directed path of $\mathfrak{o}$ starting from $v$.
For the sake of contradiction, suppose $w \notin K$. The set $E'$ of edges whose two endpoints are in $K$ and $V(G) \setminus K$ is a non-empty cut set since $G$ is connected.
By the construction of $K$, the edges in $E'$ should be all oriented toward $K$. Then any edge $e\in E'$ cannot be contained in a directed cycle.
\end{proof}

Now we are ready to prove a duality theorem between the boundary acyclic orientations and the totally cyclic orientations.

\begin{thm} \label{duality of local AO and TCO}
Let $G$ be a graph on a surface.
The map $BAO(G) \to TCO(G^{*})$ sending $\mathfrak{o}$ to $\mathfrak{o}^{*}$ is bijective.
\end{thm}
\begin{proof}
Let $G=(V,E,F)$ be a graph on a surface and $G^{*}=(F^{*},E^{*},V^{*})$ be the dual graph. We may assume that $G$ is connected. Let $V' \subset F^{*}$ be the dual of a proper subset $F' \subset F$ and $\partial V'$ be the dual of the boundary $\partial F'$.
Then $\partial V'$ is a cut set whose cut is $(V',V \setminus V')$.
Moreover, an orientation on $\partial F'$ is coherently oriented if and only if its dual orientation on $\partial V'$ is coherently oriented.
This shows that the map $\mathfrak{o} \mapsto \mathfrak{o}^{*} $ is a bijection between boundary acyclic orientations of $G$ and totally cyclic orientations of $G^*$, by Lemma \ref{lem tco characterization}.
\end{proof}


\subsection{The dual of acyclic orientations}

We turn our attention to presenting another duality theorem involving acyclic orientations. Let $G$ be a graph on a surface.
A \emph{cocycle} of $G$ is a sequence $(f_0, e_1, f_1,...,e_k,f_k=f_0)$ of non-repeating edges and faces of $G$ such that each consecutive face and edge are adjacent. We abbreviate a cocycle as its edge set. It is not hard to see that a cocycle is a dual of a cycle which explains the naming of cocycle. Topologically, a cocycle forms an annulus whose core is its dual cycle. So it is bounded by two circles. A cocycle is called \emph{coherently oriented} if the direction of each edge is from the same boundary circle toward the other. In Figure \ref{figure:kite}, $C=(f_2,e_4,f_3,e_5,f_4,e_2,f_2)$ is a cocycle. Since all of the edges of $C$ are directed from the circle formed by $e_1$ toward the circle formed by $e_3$ and $e_6$, $C$ is coherently oriented. 
A pair $(e_1, e_2)$ of edges of a cocycle $C$ is called \emph{coherently oriented with respect to $C$} if both $e_1$ and $e_2$ are directed from the same boundary circle to the other. 

A walk is a sequence $(v_0, e_1, v_1,...,e_k, v_k)$ of vertices and edges of $G$ such that each consecutive vertex and edge are adjacent. We abbreviate a walk as its edge set. A closed walk is a walk whose endpoint and starting point are the same, i.e. $v_k=v_0$. Note that edges in a walk may be repeated, while the edges in a cycle are never repeated. A directed walk is a walk with an orientation that is compatible with the sequence of edges for the walk. If $W$ is a directed walk and $C$ is a cocycle, we denote $C \cap W$ as the intersection of the undirected edge set of $W$ and $C$.
We are now ready to define the \emph{totally bi-walkable orientation} which will be shown to be the dual of the acyclic orientation.

\begin{defn}
A directed closed walk $W$ is called \emph{bidirectional} if it satisfies the following condition:
If $C$ is a cocycle such that $C \cap W \neq \emptyset$, then there is a pair $(e_1, e_2)$ of edges in $C \cap W$ that are not coherently oriented with respect to $C$.

An orientation $\mathfrak{o}$ of $G$ is called \emph{totally bi-walkable} if for every edge $e$, there exists a bidirectional closed walk $W_e$ containing $e$. Denote the set of totally bi-walkable orientations of $G$ by $TBO(G)$.
\end{defn}
It is easy to show that if there is a directed closed walk containing $\overrightarrow{e}$ then there is a directed cycle containing $\overrightarrow{e}$.
Therefore, $TBO(G)\subseteq TCO(G)$.

A cycle $\mathcal{C}$ of $G$ is called \emph{separating} if contracting $\mathcal{C}$ increases the number of components of the embedded surface, i.e., $c(G/\mathcal{C})=c(G)+1$. A cycle which is not separating is called \emph{non-separating}. In Figure \ref{figure:separating cycle}, two cycles of colored red are separating cycles and a cycle of colored blue is a non-separating cycle.
Accordingly, a cocycle has two types, the separating and the non-separating each of which is the dual of the separating and the non-separating cycle, respectively. We provide the following characterization of totally bi-walkable orientations, which is an analogue of the Lemma \ref{lem tco characterization}.

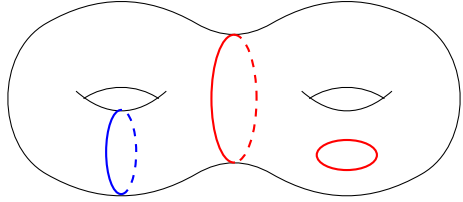
\begin{figure}
\begin{tikzpicture}
  \useasboundingbox (2,-1.5) rectangle (3,1.5);
\draw[smooth] (0,1) to[out=30,in=150] (2,1) to[out=-30,in=210] (3,1) to[out=30,in=150] (5,1) to[out=-30,in=30] (5,-1) to[out=210,in=-30] (3,-1) to[out=150,in=30] (2,-1) to[out=210,in=-30] (0,-1) to[out=150,in=-150] (0,1);
\draw[smooth] (0.4,0.1) .. controls (0.8,-0.25) and (1.2,-0.25) .. (1.6,0.1);  
\draw[smooth] (0.5,0) .. controls (0.8,0.2) and (1.2,0.2) .. (1.5,0);
\draw[smooth] (3.4,0.1) .. controls (3.8,-0.25) and (4.2,-0.25) .. (4.6,0.1);
\draw[smooth] (3.5,0) .. controls (3.8,0.2) and (4.2,0.2) .. (4.5,0);

\begin{scope}
    \clip (-3,-1) rectangle (2.5,2.5);
    \draw[thick,red] (2.5,0) ellipse (0.3 and 0.85);
\end{scope}

\begin{scope}
    \clip (2.5,-1) rectangle (5,2.5);
    \draw[thick,red,dashed] (2.5,0) ellipse (0.3 and 0.85);
\end{scope}

\begin{scope}
    \clip (1,-2) rectangle (2,2.5);
    \draw[thick,blue,dashed] (1,-0.71) ellipse (0.2 and 0.56);
\end{scope}
\begin{scope}
    \clip (-1,-2) rectangle (1,2.5);
    \draw[thick,blue] (1,-0.71) ellipse (0.2 and 0.56);
\end{scope}

\draw [thick,red] (4,-0.75) ellipse (0.4 and 0.2);

\end{tikzpicture}

  
    \caption{Separating cycles (red) and a non-separating cycle (blue)}
    \label{figure:separating cycle}
\end{figure}

\begin{lem} \label{lem TBO charaterization}
Let $G$ be a graph on a surface. Let $\mathfrak{o}$ be an orientation of $G$. Then $\mathfrak{o}$ is a totally bi-walkable orientation if and only if there is neither coherently oriented cut set nor coherently oriented non-separating cocycle, i.e. there is no coherently oriented cocycle.
\end{lem}

\begin{proof}
Only if part is clear. For the converse, it suffices to show the following by Lemma \ref{lem tco characterization}: If a totally cyclic orientation $\mathfrak{o}$ has no coherently oriented non-separating cocycle, then $\mathfrak{o}$ is a totally bi-walkable orientation.

We may assume that $G$ is connected. By the hypothesis, for each non-separating cocycle $C$ choose a pair of edges $(e_{C,1}, e_{C,2})$ which is not coherently oriented with respect to $C$. Since the graph is connected and $\mathfrak{o}$ is totally cyclic, $\mathfrak{o}$ is strongly connected by Lemma~\ref{lem tco characterization}.
Now we can conclude that there is a closed walk from any edge $e$ that contains all the edges $e_{C,1}, e_{C,2}$ for each non-separating cocycle $C$'s, which proves the claim.
\end{proof}

Now we present the duality between acyclic orientations and totally bi-walkable orientations, which completes an explanation of the bottom rectangle in Figure \ref{figure:cube}.

\begin{thm} \label{duality of AO and TBO}
Let $G$ be a graph on a surface.
The map $AO(G) \to TBO(G^{*})$ sending $\mathfrak{o}$ to $\mathfrak{o}^{*}$ is bijective.
\end{thm}

\begin{proof}
This is immediate from Theorem \ref{duality of local AO and TCO} and Lemma \ref{lem TBO charaterization} since we classified cycle type as separating or non-separating.
\end{proof}

\subsection{Applications}
A duality makes us convey one theorem to another one if the contents in the theorem have a notion of dual. We present some of the examples which are `dual' of the original theorems.

Let $G$ be a graph on a surface. For an orientation $\mathfrak{o}$ of $G$, a face $f$ of $G$ is called a \emph{cw-face} with respect to $\mathfrak{o}$ if $f$ is surrounded by clockwise oriented edges. Note that for a graph on a surface, the dual of a sink (or a source) vertex is a cw-face (or a ccw-face). There are theorems on the number of acyclic orientations with some restraints on sinks. Those can be `dualized' in terms of totally bi-walkable orientations as follow.

\begin{cor}
Let $G$ be a graph on a surface and $\tau(G^*;k)$ be the tension polynomial of the dual graph $G^*$. Then we have the following.
\begin{itemize}
    \item[(i)]
    $|\tau(G^*;-1)|$ is equal to the number of totally bi-walkable orientations of $G$.
    \item[(ii)]
    For a fixed face $f$ of $G$, the constant term of $\tau(G^*;k)$ is equal to the number of totally bi-walkable orientations of $G$ with unique cw-face $f$.
    \item[(iii)]
    Suppose $X_{G^*}=\sum_{\lambda} c_\lambda e_\lambda$, is the expansion of the chromatic symmetric function $X_{G^*}$ of $G^*$ in terms of elementary symmetric function $e_\lambda$. Let $\mathfrak{tbo}(G,j)$ be the number of totally bi-walkable orientations of $G$ with $j$ cw-faces. Then we have
    \[\mathfrak{tbo}(G,j)=\sum_{l(\lambda)=j}c_\lambda,\]
    where $l(\lambda)$ denotes the length of a partition $\lambda$ (See \cite{Sta95} for undefined terms).
    \item[(iv)]
    The generating function for totally bi-walkable orientations of $G$ along with their number of cw-faces is given by
    \begin{equation*}\label{sink generating function}
    \sum_{j\ge 0} \mathfrak{tbo}(G,j)q^j    =\sum_{S \subseteq E}(-1)^{|E(S)|-|V(G)|+c(S)}{\prod_{C \in \mathcal{C}(S)}{\left(1-(1-q)^{|V(C)|}\right)}},
    \end{equation*} 
where $\mathcal{C}(S)$ is the set of connected components of the spanning subgraph of $G^*$ induced by $S \subseteq E$ by regarding the vertices not belong to $S$ as singleton vertices in the subgraph.
\end{itemize}
\end{cor}
\begin{proof} 
(i), (ii), (iii) and (iv) are the dual statements of \cite[Corollary 1.3]{Sta73}, \cite[Theorem 7.3]{GZ83}, \cite[Theorem 3.3]{Sta95}, and \cite[Theorem 3.2]{HJLOY19}, respectively.
\end{proof}


One can obtain other expressions for $\sum_{j\ge 0} \mathfrak{tbo}(G,j)q^j$ in terms of non-broken circuits, bond lattice, or acyclic orientations (or totally bi-walkable orientations) by combining the duality with \cite[Theorem 3.5]{HJLOY19}, \cite[Theorem 3.10]{HJLOY19}, or \cite[Theorem 3.14]{HJLOY19}, respectively.


\section{The reciprocities}\label{section4}
\subsection{The reciprocity for local tensions}
Let $G$ be a graph on a surface with an orientation $\sigma$. A \emph{local} $\mathbb{Z}_k$-tension of $G$ is an assignment $t:E\rightarrow\mathbb{Z}_k$ such that for each face $f$, the sum of values of edges in the boundary of $f$ with respect to $\sigma$, i.e.
\[\sum_{e\in \partial(f)}\overrightarrow{\partial_\sigma f}(e) t(e)=0.\]
One might notice that a local tension is the `1-cocycle' of the chain complex given by $G$ (the word cocycle is different from what we used in Section \ref{section3}). Let $\tau^{loc}(G;k)$ denotes the number of nowhere-zero local $\mathbb{Z}_k$-tension and call it as \emph{local tension polynomial} of $G$ \cite{GKRV18}. 

The purpose of this subsection is to prove a reciprocity theorem for the local tension polynomial. We recall a reciprocity theorem for flow polynomials of graphs \cite{BS12}.
\begin{thm}\cite[Theorem 2.2]{BS12}\label{BS12}
Let $G=(V,E)$ be an oriented graph and $k$ be a positive integer. For the flow polynomial $\phi(G;k)$ of $G$, $(-1)^{|E|-|V|+c(G)}\phi(G;-k)$ is equal to the number of pairs $(f,\mathfrak{o})$ where $f$ is a $\mathbb{Z}_k$-flow on $G$ and $\mathfrak{o}$ is a totally cyclic orientation of $G\sslash supp(f)$.
\end{thm}

In the above theorem, the symbol $\sslash$ denotes the contraction for abstract graphs, not the contraction for ribbon graphs (the contraction of a loop for an abstract graph is given by a deletion in the sense of a ribbon graph). For a graph on a surface $G=(V,E,F)$ and $A=\{e_1,\dots,e_l\}\subseteq E$, let us define $\ssetminus$ by \[G\ssetminus A=(\cdots((G*_1e_1)*_2e_2)\cdots*_le_l),\] where $*_k$ is a deletion if $e_k$ is not a coloop at the step and $*_k$ is a contraction if otherwise. By Proposition \ref{del-con dual}, one can see that the operation $\ssetminus$ is the dual operation of $\sslash$. Now we state the main theorem of this section, which is the dual version of Theorem \ref{BS12}.

\begin{thm}\label{reciprocity for local tension}
Let $G=(V,E,F)$ be an oriented graph on a surface and $k$ be a positive integer. Then $(-1)^{|E|-|F|+c(G)}\tau^{loc}(G;-k)$ counts pairs $(t,\mathfrak{o})$ where $t$ is a local $\mathbb{Z}_k$-tension on $G$ and $\mathfrak{o}$ is a boundary acyclic orientation of $G\ssetminus supp(t)$. 
\end{thm}

Applying the duality between boundary acyclic orientation and totally cyclic orientation (Theorem \ref{duality of local AO and TCO}) and the duality between local tension and flow to Theorem \ref{BS12}, Theorem \ref{reciprocity for local tension} follows immediately. Although this approach provides one proof, we rather prove Theorem \ref{reciprocity for local tension} directly by following the scheme of \cite{BS12} using the Ehrhart theory. For details on the Ehrhart theory, see \cite{Sta86}.

Let $G=(V,E,F)$ be a graph on a surface with an orientation $\sigma$. We define the \emph{face matrix} $D(G)\in \{0,\pm 1\}^{{F}\times E}$ by
\[
D(G)_{(f,e)}=  \overrightarrow{\partial_\sigma f}(e).
\]
By identifying $\mathbb{Z}_k$ with $\{0,1,\dots,k-1\}$, we can identify a nowhere-zero local $\mathbb{Z}_k$-tension with $t\in \mathbb{Z}^E$ satisfying $0<t(e)<k$ and $D(G)\cdot t=k\cdot d$ for some $d\in \mathbb{Z}^{F}$. For $b\in\mathbb{Z}^F$, let
$$
\mathring{P}_G(b):=\{p\in \mathbb{R}^E: D(G)\cdot p=b, 0<p(e)<1, \text{ for all } e\in E \}.
$$
For a boundary $C$, there is a surface $\Sigma'\subseteq \Sigma$ which bounds $C$. Notice that the sum of values assigned to $C$ could be obtained by the sum of values assigned to the boundary of faces contained in $\Sigma'$. 
Therefore, the points in $\left(k\cdot \mathring{P}_{G}(b)\right) \cap \mathbb{Z}^E$ represent nowhere-zero local $\mathbb{Z}_k$-tensions, where $k\cdot P$ is a dilation of $P$ by $k$. Denote the set of \emph{feasible} $b$'s by $\mathcal{F}_{G}=\{ b\in \mathbb{Z}^F: \mathring{P}_{G}(b)\neq\emptyset \}$. We remark that the set $\mathcal{F}_{G}$ is finite and $\mathring{P}_{G}(b)$ and $\mathring{P}_{G}(b')$ are disjoint for distinct $b$ and $b'$ in $\mathcal{F}_G$.

\begin{ex}\label{Example: face matrix}
For a graph on a surface $G$ with an orientation $\sigma$ in Figure \ref{figure:graph on surface}, there are two faces $f_1$ and $f_2$, and four edges $e_1, e_2, e_3$ and $e_4$ with a given orientation.

\begin{figure}[h]
    \centering
    \includegraphics[width=0.5\textwidth]{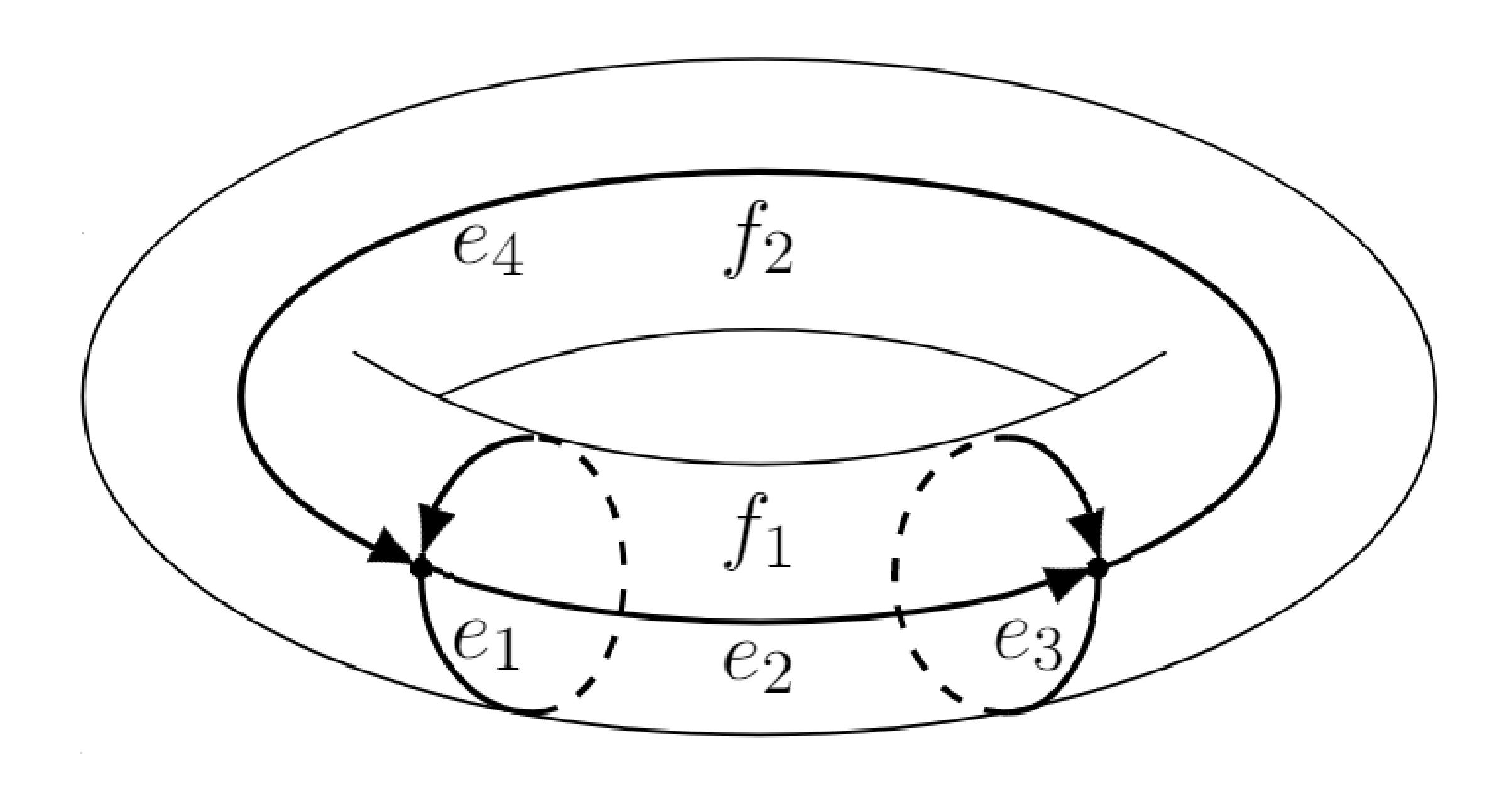}
    \caption{A graph on a surface $G$ with an orientation $\sigma$.}
    \label{figure:graph on surface}
\end{figure}


And we have the following face matrix of $G$.
\[D(G) =
  \begin{array}{@{}c@{}}
    \rowind{$f_1$} \\ \rowind{$f_2$}
  \end{array}
  \mathop{\left[
  \begin{array}{ *{5}{c} }
     \colind{1}{$e_1$}  &  \colind{0}{$e_2$}  &  \colind{1}{$e_3$}  & \colind{0}{$e_4$} \\
    -1  & 0 &  -1  & 0 \\
  \end{array}
  \right]}^{
  \begin{array}{@{}c@{}}
   \\ \mathstrut
  \end{array}}\]

\end{ex}

A \emph{(closed convex) polytope} $P$ is a convex hull of finite set of points (vertices) in $\mathbb{R}^l$. For a given polytope $P\subseteq \mathbb{R}^l$, the \emph{Ehrhart polynomial} is defined by $$\mathbf{Ehr}(P;k)=|(k\cdot P)\cap \mathbb{Z}^l|.$$  A polytope $P\subseteq \mathbb{R}^l$ is called \emph{lattice} (or \emph{integral}) if all of its vertices have integral coordinates and \emph{rational} if all of its vertices have rational coordinates. Ehrhart showed that for lattice polytope $P$, $\mathbf{Ehr}(P;k)$ is a polynomial, which verifies the name of Ehrhart polynomial. The following proposition is straightforward by the construction of $\mathring{P}_G(b)$ and $\mathcal{F}_G$.

\begin{prop}\label{local tension Ehrhart}
Let $G=(V,E,F)$ be an oriented graph on a surface. Then
$$\tau^{loc}(G;k)=\sum_{b\in\mathcal{F}_{G}}\mathbf{Ehr}(\mathring{P}_{G}(b);k).$$
\end{prop}

The most elegant application of the Ehrhart theory may be the \emph{Ehrhart-Macdonald reciprocity}, which is given in the following theorem.
\begin{thm}\label{Ehrhart}(Ehrhart-Macdonald reciprocity)
Let $P$ be a rational polytope and $\mathring{P}$ be its interior. Then
$$(-1)^{dimP}\mathbf{Ehr}(\mathring{P};-k)=\mathbf{Ehr}(P;k).$$
\end{thm}

By the Ehrhart-Macdonald reciprocity theorem and Proposition \ref{local tension Ehrhart}, it is clear that
\begin{equation}\label{tau ehrhart}
(-1)^{dim P}\tau^{loc}(G;-k)=\left|\bigcup_{b\in\mathcal{F}_G} k \cdot P_G(b)\cap \mathbb{Z}^E\right|,
\end{equation}
where $P_G(b)$ is the closure of $\mathring{P}_G(b)$.
Therefore, to prove Theorem \ref{reciprocity for local tension} it is sufficient to construct a bijection between the set in the right-hand side of Equation (\ref{tau ehrhart}) and the set of pairs described in Theorem \ref{reciprocity for local tension}. 

For a graph $G$ with an orientation $\sigma$ on a surface, the kernel of its face matrix \[ker D(G)=\{p\in\mathbb{R}^E: D(G)\cdot p=0\}\] defines a subspace of $\mathbb{R}^E$. Each coordinate hyperplane $H_e=\{p\in\mathbb{R}^E: p(e)=0\}$ intersects with $ker D(G)$ as a hyperplane in $ker D(G)$ except for the case when $e$ is a contractible loop. We abuse our notation $H_e$ to denote the intersection of $H_e$ with $ker D(G)$. We consider a set 
\begin{equation}\label{Equation: H_G}
    \mathcal{H}_G :=ker D(G) \setminus \bigcup_{e\in E} H_e.
\end{equation}
Note that if $G$ has a contractible loop then $\mathcal{H}_G$ is empty, and if $G$ has no contractible loop then $\mathcal{H}_G$ is a \emph{hyperplane arrangement}, i.e. a finite set of hyperplanes containing the origin in $\mathbb{R}^l$ for some natural number $l$. An open region of $\mathcal{H}$ is a connected component of $\mathbb{R}^l\setminus \cup \mathcal{H}$. 
The following lemma plays a key role in our proof of the main theorem.

\begin{lem}\label{boundary acyclic orientation}
Let $G=(V,E,F)$ be a graph on a surface with an orientation $\sigma$. For each open region $R$ in $\mathcal{H}_G$, choose a point $p\in R$. Let  $\sigma(R)$ be an orientation of $G$ obtained from $\sigma$ by reversing the directions of edge $e$'s with $p(e)<0$.
Then the map
\begin{equation*}\label{Equation: the map sigma(R)}
    R \mapsto \sigma(R)
\end{equation*}
is a bijection between open region of $\mathcal{H}_G$ and boundary acyclic orientations of $G$.
\end{lem}
\begin{proof}
Since moving $p$ continuously in the region $R$ does not change the sign of $p(e)$, the map $R\mapsto \sigma(R)$ is well-defined.
Let $R$ be an open region of $\mathcal{H}_G$ and $p$ be a point in $R$, and $F'$ be a subset of $F$.
Since $p\in\text{ker}D(G)$, 
\begin{equation}\label{eq: boundary sum}
    \sum_{e \in \partial F'} \overrightarrow{\partial_{\sigma} F'}(e) p(e)
    =\sum_{f \in F'} \sum_{e \in \partial f} \overrightarrow{\partial_{\sigma} f}(e) p(e) = 0
\end{equation}
By definition of $\sigma(R)$, the signs of $\overrightarrow{\partial_{\sigma(R)} F'}(e)$ and $\overrightarrow{\partial_{\sigma} F'}(e) p(e)$ are the same for all $e \in E$.
By \eqref{eq: boundary sum}, the nonzero entries of $\overrightarrow{\partial_{\sigma(R)} F'}$ are not all $1$ or $-1$. Since $F' \subseteq F$ is arbitrary, $\sigma(R)$ is a boundary acyclic orientation.


Recall that the boundary acyclic orientations are dual of the totally cyclic orientations (Theorem~\ref{duality of local AO and TCO}), 
and the 
(coherently oriented) cocycles are the dual of (directed) cycles. Therefore, an orientation $\mathfrak{o}$ is a boundary acyclic orientation of $G$ if and only if there exists a coherently oriented cocycle containing $e$ for each edge $e$ in $G$.

Let $\mathfrak{o}$ be a boundary acyclic orientation of $G$. For each coherently oriented cocycle $C$ in $\mathfrak{o}$, assign a vector $\overrightarrow{C}\in \mathbb{Z}^E$ by
\[
\overrightarrow{C}(e):=
\begin{cases}
0, & \text{ if } e\notin C \\
1, & \text{ if the orientation of } e \text{ in } \mathfrak{o} \text{ and } \sigma \text{ agree,}\\ 
-1, & \text{ if the orientation of } e \text{ in } \mathfrak{o} \text{ and } \sigma \text{ do not agree.}
\end{cases}
\]
Note that $D(G)\cdot \overrightarrow{C}=0$, by definition of the face matrix and cocycle.
Let us define a vector $t\in \mathbb{R}^E$ by
\[
p_\mathfrak{o}=\sum_{C} \overrightarrow{C},
\]
where the summation is over coherently oriented cocycle $C$ in $\mathfrak{o}$. Since each vector $\overrightarrow{C} \in ker D(G)$, it is clear that $p_\mathfrak{o}\in ker D(G)$. On the other hand, by the characterization of boundary acyclic orientation, one concludes that $p_\mathfrak{o}(e)\neq0$ for each $e\in E$. Moreover, $p_\mathfrak{o}(e)>0$ if the orientation of $e$ in $\mathfrak{o}$ and $\sigma$ agrees, and $p_\mathfrak{o}(e)<0$ otherwise. It is easy to show that the region $R$ containing $p_\mathfrak{o}$ satisfies $\sigma(R)=\mathfrak{o}$.
\end{proof}
\begin{rem}
In \cite{GZ83}, two lemmas are similar to Lemma~\ref{boundary acyclic orientation}. One is \cite[Lemma 7.1]{GZ83} which gives a bijection between open regions of a hyperplane arrangement associated with a graph $G$ and acyclic orientations of $G$. The other one is \cite[Lemma 8.1]{GZ83} which gives a bijection between open regions of a hyperplane arrangement associated to a graph $G$ and totally cyclic orientations of $G$.
The statement of Lemma~\ref{boundary acyclic orientation} is parallel with \cite[Lemma 7.1]{GZ83}. However, the proof is quite similar to the argument of \cite[Lemma 8.1]{GZ83}.
\end{rem}

\begin{ex}
In our running example (Example~\ref{Example: face matrix}), \[\mathcal{H}_G = \mathbb{R}_{e_2}\oplus \mathbb{R}_{e_4}\oplus \mathbb{R}_{e_1 - e_3} \setminus \bigcup_{e \in E} H_e\] has $8$ connected components. For example, $R=\{(a,b,a,-c): a,b,c>0\}$ is an open region of $ker D(G)\setminus \bigcup_{e\in E}H_e$. The orientation $\sigma(R)$ obtained from given orientation $\sigma$ by reversing orientation of $e_3$ and $e_4$ is a boundary acyclic orientation. 
\end{ex}
We now give a proof of Theorem \ref{reciprocity for local tension}.

\begin{proof}(proof of Theorem \ref{reciprocity for local tension})
Let $k$ be a positive integer and 
$$
P_G(b):=\{p\in \mathbb{R}^E \mid D(G)\cdot p =b, 0\le p(e) \le 1, \text{ for all } e\in E\}
$$
for a feasible $b$, which is the closure of $\mathring{P}_G(b)$.
By Theorem~\ref{Ehrhart}, it suffices to show that there is a bijection   
$$
\phi : \bigcup_{b\in\mathcal{F}_G} k \cdot P_G(b)\cap \mathbb{Z}^E
\rightarrow
\{(f,\mathfrak{o}) \mid f \text{ is a local } \mathbb{Z}_k\text{-tension on } G, \mathfrak{o} \in BAO(G\ssetminus supp(f)) \}.
$$

Fix $t \in \bigcup_{b\in\mathcal{F}_G} k \cdot P_G(b)\cap \mathbb{Z}^E$. Let $\bar{t}$ be the local $\mathbb{Z}_k$-tension corresponding to $t$, which is given by $t$ modulo $k$.
Note that $t$ induces a set partition $E=E_0(t) \cup  supp(\bar{t}) \cup E_{1}(t)$, where 
\[
E_0(t)=\{e \in E \mid t(e)=0 \}, E_1(t)=\{e \in E \mid t(e)=k \}, \text{ and } supp(\bar{t})=\{e\in E \mid \bar{t}(e)\neq 0\}.
\]
Let us fix a point $z_b$ in $\mathring{P}_G(b)$ and set $z_t:=z_b-t$. And, let  $z_t|_{E_0(t)\cup E_1(t)}$ denotes the restriction of $z_t$ by only considering coordinates corresponding to $E_0(t) \cup E_1(t)$. We claim that 
$$
z_t|_{E_0(t)\cup E_1(t)} \in ker D(G\ssetminus supp(\bar{t})) \setminus \bigcup_{e\in E_0(t)\cup E_1(t)}H_e.
$$
To prove the claim, let us describe $D(G\ssetminus e)$ for $e \in E$. For the first case, suppose $e$ is not a coloop and there are faces $f_1$ and $f_2$ that are incident to $e$. Then $F(G\setminus e)$ consists of $F(G)\setminus\{f_1, f_2\}$ and a new face $f$ that is topologically union of $f_1$, $f_2$ and $e$ in $G$. In this case, $D(G\setminus e)$ is given by replacing two rows indexed by $f_1$ and $f_2$ with a new row indexed by $f$ which is the sum of rows indexed by $f_1$ and $f_2$ and then deleting the column indexed by $e$ (zero column). For the second case, suppose $e$ is a coloop and $e$ is incident to a face $f$ whose boundary contains $e$ twice. Then $F(G/e)$ consists of $F(G)\setminus\{f\}$ and a new face $f'$ whose boundary coincide with that of $f$. In this case $D(G/e)$ is given by deleting a column indexed by $e$ (zero column) and replacing a row index $f$ by $f'$. By the descriptions above and the fact that $z_t=z_b - t \in ker D(G)$, $z_t|_{E_0(t)\cup E_1(t)}$ is in $ker D(G \ssetminus supp(\bar{t}))$. Since the $e$-coordinates of $z_b$ and $t$ are distinct for each $e \in E_0(t) \cup E_1(t)$, we conclude that the claim holds.


Now we are ready to define a map $\phi$.
Let $R$ be the region of 
$$
ker D(G\ssetminus supp(\bar{t}))\setminus \bigcup_{e\in E_0(t)\cup E_1(t)}H_e
$$
containing $z_t|_{E_0(t)\cup E_1(t)}$. Let us set $\sigma(t) := \sigma(R)$, where $\sigma(R)$ is a boundary acyclic orientation of $G\ssetminus supp(\bar{t})$ as given in Lemma~\ref{boundary acyclic orientation}. Let $\phi$ be the map defined by

$$
\phi(t)=(\bar{t}, \sigma(t)).
$$

Let us construct the inverse of $\phi$.
For a local $\mathbb{Z}_k$-tension $t'$ and a boundary acyclic orientation $\mathfrak{o}$ of $G\ssetminus supp({t'})$, define $\psi(t',\mathfrak{o}) \in \mathbb{Z}^E$ by
$$
\psi(t',\mathfrak{o})(e):=
\begin{cases}
t'(e), & \text{   if } t'(e) \neq 0 \\
0, & \text{      if } t'(e)=0 \text{ and $\sigma$ agrees with }\mathfrak{o}\\
k, & \text{      if } t'(e)=0 \text{ and $\sigma$ disagrees with }\mathfrak{o}.
\end{cases}
$$
Note that $\phi\circ\psi(t',\mathfrak{o})=(t',\mathfrak{o})$, and the point $\psi(t',\mathfrak{o})$ is in the boundary of $k \cdot P_{G}(b')$ where $b'=D(G)\cdot \psi(t',\mathfrak{o})$. Therefore, it remains to show that $b'$ is feasible. By Lemma \ref{boundary acyclic orientation} again, there is a vector $x\in \mathbb{R}^E$ with $D(G)\cdot x=0$ and $x(e)<0$ for $e$ whose given orientation agrees with $\mathfrak{o}$ and $x(e)>0$ for $e$ whose given orientation disagrees with $\mathfrak{o}$. Then $(\psi(t',\mathfrak{o})+\epsilon x)$ is in $k \cdot \mathring{P}_G(b')$ for sufficiently small $\epsilon>0$ which means that $\mathring{P}_G(b')$ is non-empty, thus $b'\in \mathcal{F}_{G}$.

\end{proof}

To attain our main goal, which is explaining dualities and reciprocities depicted in Figure~\ref{figure:cube}, we define one more notion.
Recall that a $\mathbb{Z}_k$-flow of an oriented graph $G$ is an assignment of elements in $\mathbb{Z}_k$ to each edge such that the sum of values assigned to the edges adjacent to any vertex with respect to a given orientation is zero. It is equivalent to check the condition for each cutset not only the edges adjacent to a vertex (cut set associated with a vertex). The \emph{balanced} $\mathbb{Z}_k$-flow is a $\mathbb{Z}_k$-flow such that for each cocycle $C$, the values assigned to the edges of $C$ sum up to zero. For a graph $G$ on a surface, let $\phi^{bal}(G;k)$ be the number of nowhere-zero balanced $\mathbb{Z}_k$-flow. Now we state all the dualities and reciprocities described in Figure \ref{figure:cube}.

\begin{thm}\label{thm:main}
Let $G$ be an oriented graph on a surface and $k$ be a positive integer. There are dualities as follows.
\begin{align*}
    \tau(G;k)&=\phi^{bal}(G^*;k),&
    \tau^{loc}(G;k)&=\phi(G^*;k),\\
    |AO(G)|&=|TBO(G^*)|,&
    |BAO(G)|&=|TCO(G^*)|.
\end{align*}
And they are related with the following identities.
\begin{align*}
    |\tau(G;-1)|&=|AO(G)|,&
    |\tau^{loc}(G;-1)|&=|BAO(G)|,\\
    |\phi(G^*;-1)|&=|TCO(G^*)|,&
    |\phi^{bal}(G^*;-1)|&=|TBO(G^*)|.
\end{align*}
Moreover, they are related via reciprocity as follows.
\begin{align*}
    |\tau(G;-k)|&=|\{(t,\mathfrak{o}): t\in T(G;k) \text{ and }\mathfrak{o}\in AO(G\setminus supp(t))\}|,\\
    |\tau^{loc}(G;-k)|&=|\{(t,\mathfrak{o}): t\in T^{loc}(G;k) \text{ and }\mathfrak{o}\in BAO(G\ssetminus supp(t))\}|,\\
    |\phi(G;-k)|&=|\{(f,\mathfrak{o}): f\in \Phi(G;k) \text{ and }\mathfrak{o}\in TCO(G\sslash supp(t))\}|,\\
    |\phi^{bal}(G;-k)|&=|\{(f,\mathfrak{o}): f\in \Phi^{bal}(G;k) \text{ and }\mathfrak{o}\in TBO(G\slash supp(t))\}|,
    \end{align*}
where $T(G;k)$ ($T^{loc}(G;k)$, $\Phi(G;k)$ and $\Phi^{bal}(G;k)$, respectively) denotes the set of $\mathbb{Z}_k$-tensions (local tensions, flows and balanced flows, respectively).
\end{thm}
\begin{proof}

All that remains to show is the first and the last equality.
One can directly obtain the duality between balanced flows and the tensions on graphs on surfaces from their definitions. This proves the first equality.
Combining this duality and the reciprocity \cite[Theorem 4.2]{BS12} yields the last equality.

\end{proof}
\begin{rem}
In~\cite{GKRV18}, Goodall et al. defined a Tutte polynomial for graphs on orientable surfaces and showed that the local tension polynomial and the flow polynomial can be obtained as specializations. Moreover, it is not hard to get the tension polynomial and the balanced flow polynomial as specializations of their Tutte polynomial. Therefore, the combinatorial objects in Theorem~\ref{thm:main} can be counted as specializations of the Tutte polynomial of \cite{GKRV18}.

Throughout this paper, graphs on surfaces are assumed to be orientable. In \cite{GLRV20}, a Tutte polynomial for graphs on non-orientable surfaces is defined. They showed that their Tutte polynomial specializes to the local tension and the `local flow' (See~\cite{GLRV20} for the definition) polynomial for graphs on (general) surfaces. It would be interesting to investigate combinatorial objects that can be counted as specializations of the Tutte polynomial of \cite{GLRV20} generalizing our work in this paper.

\end{rem}
\subsection{Inside-out polytope and reciprocity for integral local tensions}\label{section4.2}
For a graph $G$, an \emph{integral $k$-flow} of $G$ is a $\mathbb{Z}$-flow of $G$ with its absolute values are less than $k$. Tutte \cite{Tut50, Tut54} proved that nowhere-zero $\mathbb{Z}_k$-flow exists if and only if nowhere-zero integral $k$-flow exists. The number of nowhere-zero $\mathbb{Z}_k$-flows however, is not equal to the number of nowhere-zero integral $k$-flows. More recently, Kochol \cite[Theorem 1]{Koc02} showed that the number of nowhere-zero integral $k$-flows is a polynomial in $k$. We call it the \emph{integral $k$-flow polynomial} and denote it by $\phi_{\mathbb{Z}}(G;k)$. Inspired by Kochol's work, Beck and Zaslavsky developed a theory of inside-out polytope \cite{BZ06} and provided a reciprocity theorem for integral $k$-flows \cite[Theorem 3.1]{BZ06JCT}.

For a graph $G$ on a surface with an orientation $\sigma$, we define an \emph{integral k-local tension} to be a mapping $t:E\overrightarrow{} \{-(k-1),\dots,-1,0,1,\dots,k-1\}$ such that for each face $f$, $$\sum_{e\in \partial f} \overrightarrow{\partial_\sigma f}(e)t(e)=0.$$ 
Let $\tau^{loc}_\mathbb{Z}(G;k)$ be the number of nowhere-zero integral $k$-local tensions of $G$. An integral $k$-local tension $t$ and a  boundary acyclic orientation $\mathfrak{o}$ is called \emph{compatible} if $\mathfrak{o}$ agrees with $\sigma$ for edges $e$ with $t(e)>0$, disagrees for edges $e$ with $t(e)<0$.
In this subsection, we provide a reciprocity theorem for integral local tension as follows.

\begin{thm}\label{thm: reciprocity for integral local tension}
Let $G$ be an oriented graph on a surface. Then for $k\ge 0$,
$|\tau^{loc}_{\mathbb{Z}}(G;-k)|$ equals to the number of pairs $(t,\mathfrak{o})$, where $t$ is an integral $(k+1)$-local tensions and $\mathfrak{o}$ is a  boundary acyclic orientation compatible with $t$.

In particular, the absolute value of the constant term $\tau^{loc}_\mathbb{Z}(0)$ equals to the number of boundary acyclic orientations.
\end{thm}


To prove the above theorem, we recall the theory of inside-out polytopes studied in \cite{BZ06}. 
A hyperplane arrangement $\mathcal{H}$ is called \emph{rational} if each hyperplane in $\mathcal{H}$ has a rational normal vector. 

A pair $(P,\mathcal{H})$ is called a \emph{rational inside-out polytope} if $P$ is a rational convex polytope and $\mathcal{H}$ is a rational hyperplane arrangement. An open region of $(P,\mathcal{H})$ is a nonempty intersection of an open region of $\mathcal{H}$, and the interior $\mathring{P}$ of $P$. Its closure is called a closed region of $(P,\mathcal{H})$.

The \emph{(closed) Ehrhart quasipolynomial} of an inside-out polytope $(P,\mathcal{H})$ is defined by
$$\mathbf{Ehr}((P,\mathcal{H}),k)=\sum_{x\in \mathbb{Z}^l\cap k\cdot P}m_{P,\mathcal{H}}(x),$$
where $m_{P,\mathcal{H}}(x)$ is the number of closed regions of $(P,\mathcal{H})$ containing $x$. The \emph{open Ehrhart quasipolynomial} is defined by
$$\mathbf{Ehr}((P,\mathcal{H})^\circ,k)=|(k\cdot \mathring{P}\setminus\cup\mathcal{H})\cap \mathbb{Z}^l|.$$ 
The following theorem is an analogue of the Ehrhart--Macdonald reciprocity theorem (Theorem~\ref{Ehrhart}).

\begin{thm}\cite[Theorem 4.1]{BZ06}\label{Ehrhart for Inside-out}
If $(P,\mathcal{H})$ is a closed, rational inside-out polytope, then  $\mathbf{Ehr}((P,\mathcal{H}),k)$ and $\mathbf{Ehr}((P,\mathcal{H})^\circ,k)$ are quasipolynomials in $k$, and $\mathbf{Ehr}((P,\mathcal{H}),0)$ equals to the number of regions in $(P,\mathcal{H})$.

Furthermore, 
$$\mathbf{Ehr}((P,\mathcal{H})^\circ,k)=(-1)^{dim P}\mathbf{Ehr}((P,\mathcal{H}),-k).$$
\end{thm}

Let $G=(V,E,F)$ be a graph on a surface with an orientation $\sigma$. 
Now we show that $\tau^{loc}_\mathbb{Z}(k)$ is an Ehrhart quasipolynomial of an inside-out polytope.
Let 
\[
\mathcal{H}_G=ker D(G) \setminus \bigcup_{e\in E} H_e
\]
as in Equation~\eqref{Equation: H_G} and
$P$ be the polytope given by 
\[
P= [-1,1]^E\cap ker D(G).
\] 
An integral $k$-local tension corresponds to a point $x\in ker D(G) \cap \mathbb{Z}^E$ such that $\frac{1}{k}x\in \mathring{P}$. Furthermore, this point $x$ is nowhere-zero if and only if $\frac{1}{k}x\notin \cup_{e\in E} H_e$. Therefore, we have
$$\tau^{loc}_\mathbb{Z}(k)=\mathbf{Ehr}((P,\mathcal{H})^\circ,k).$$
We are now in a position to give a proof of Theorem~\ref{thm: reciprocity for integral local tension}.

\begin{proof}(of Theorem~\ref{thm: reciprocity for integral local tension})
Since $P= [-1,1]^E\cap ker D(G)$, there is a natural correspondence between open regions of $\mathcal{H}_G$ and open regions of $(P,\mathcal{H}_G)$, namely
$$
R \mapsto R \cap [-1,1]^E.
$$
Combining this correspondence with Theorem~\ref{Ehrhart for Inside-out},  $|\tau^{loc}_{\mathbb{Z}}(G;-k)|$ is equal to the number of pairs $(x, R)$ where $x$ is an integral $(k+1)$-local tension and $R$ is an open region of $\mathcal{H}_G$ such that $x \in \bar{R}$, where $\bar{R}$ is the closed region of $R$.
For an open region $R$ of $\mathcal{H}_G$, observe that $x$ and the boundary acyclic orientation $\sigma(R)$ (defined in Lemma~\ref{boundary acyclic orientation}) is compatible if and only if $x \in \bar{R}$.
This proves the first assertion. 

There is only one integral $1$-local tension of $G$, $(0,0, \dots, 0)\in \mathbb{Z}^E$. Since it is compatible with any boundary acyclic orientation, the second assertion follows.
\end{proof}
One can also get a reciprocity theorem for ``integral balanced flows" by a similar argument exploiting Ehrhart theory for Inside-out polytopes.

\section*{Acknowledgments}
The research of W.-S. Jung was supported by the National Research Foundation of Korea (NRF) Grant funded by the Korean Government (NRF-2020R1F1A1A01071055). The research of J. Oh was supported by Basic Science Research Program through the National Research Foundation of Korea (NRF) funded by the Ministry of Education (NRF-2020R1A6A3A13076804).


\end{document}